\documentclass{amsart}
\usepackage[all]{xy}
\usepackage{amsmath}
\usepackage{amssymb}
\usepackage{amsthm}
\usepackage{amscd}

\newcommand{\R}{\mathcal{R}}

\newcommand{\RRR}{\mathfrak{R}}
\newcommand{\MM}{\mathfrak{M}}
\newcommand{\mm}{\mathfrak{m}}

\newcommand{\V}{\mathcal{V}}

\newcommand{\Z}{\Bbb Z}

\newcommand{\RR}{\Bbb R}
\newcommand{\TT}{\Bbb T}

\newtheorem{conj}{Conjecture}
\newtheorem{theorem}{Theorem}
\newtheorem{definition}[theorem]{Definition}
\newtheorem{proposition}[theorem]{Proposition}
\newtheorem{cor}[theorem]{Corollary}
\newtheorem{lemma}[theorem]{Lemma}
\newtheorem{problem}[theorem]{Problem}
\newtheorem{remark}[theorem]{Remark}

\DeclareMathOperator{\Span}{Span}

\DeclareMathOperator{\supp}{supp}

\DeclareMathOperator{\N}{\mathbb{N}}

\numberwithin{theorem}{section}

\begin{document}
\title[Variations on a Theorem of Bourgain]{Polynomial Ergodic Averages Converge Rapidly: Variations on a Theorem of Bourgain}
\author{Ben Krause}
\address{UCLA Math Sciences Building\\
         Los Angeles,CA 90095-1555}
\email{benkrause23@math.ucla.edu}
\date{\today}
\maketitle

\begin{abstract}
Let $L^2(X,\Sigma,\mu,\tau)$ be a measure-preserving system, with $\tau$ a $\Z$-action. In this note, we prove that the ergodic averages along integer-valued polynomials, $P(n)$,
\[ M_N(f):= \frac{1}{N}\sum_{n \leq N} \tau^{P(n)} f \]
converge pointwise for $f \in L^2(X)$. We do so by proving that, for $r>2$, the \emph{$r$-variation}, $\V^r(M_N(f))$, extends to a bounded operator on $L^2$.
We also prove that our result is sharp, in that $\V^2(M_N(f))$ is an unbounded operator on $L^2$.
\end{abstract}

\section{Introduction}
Let $(X,\Sigma,\mu)$ be a non-atomic probability space, equipped with $\tau$ a measure-preserving $\Z$-action
\[ \tau_y f(x):= f(\tau_{-y} x). \]

For $(E_i) \subset \Z$, define the averaging operators
\[ M_i f(x):= \frac{1}{|E_i|} \sum_{y\in E_i} (\tau_yf)( x);\]
the classical ($L^2$-) pointwise ergodic theorem of Birkhoff \cite{BI} says that if \[ E_i= [0,i) \subset \Z,\] then the one-dimensional averages
$\{ M_i f(x) \}$
converge pointwise $\mu$-almost everywhere for $f \in L^2(X,\Sigma,\mu)$.

A standard proof proceeds by way of a density argument:
one begins with the the dense subset
\[ \{ \phi \in L^2 \cap L^\infty : \tau \phi = \phi \} \oplus \Span\{ h - \tau h : h \in L^{\infty} \} \subset L^2 \]
on which convergence holds, and absorbs small errors using the $L^2$-boundedness of the maximal function
\[ f \mapsto \sup_i |M_if|.\]

This density argument relies crucially on the \emph{smoothness} of the intervals $[0,i)$.

Obtaining pointwise convergence results of $\{ M_i f \}$ for rougher, exotic $\{E_i \} \subset \Z$ does not necessarily follow from quantitative estimates on an appropriate maximal function, since the dense-subclass result is often unavailable in this setting.

Perhaps the most famous instance of this difficulty arose in the study of averages along the squares, i.e.\
\[ E_i := \{1,4,9,\dots, i^2\} \subset \Z.\]
Indeed, to prove pointwise convergence of the ergodic averages of $L^2$-functions along the squares, Bourgain \cite{B0} attacked the issue of oscillation more directly, by showing that an appropriate
\emph{oscillation operator} was $L^2$ ``controlled'' (\cite[\S 7]{B0}.) In a redux of his argument \cite{B1}, Bourgain did so with the assistance of the \emph{$r$-variation operators} (below), classically used in probability theory to gain quantitative information on the rates of convergence.
\begin{definition}
For a collection of functions $\{f_i\}$
\[ \V^r(f_i)(x):= \sup_{(i_k) \text{ increasing}} \left( \sum_k |f_{i_k} - f_{i_{k+1}}|^r \right)^{1/r}(x)\]
is the \emph{$r$-variation} of the $\{f_i\}$.
\end{definition}

These variation operators are more difficult to control than the maximal function $\sup_i |f_i|$: for any $j$, one may pointwise dominate
\[ \sup_i |f_i | \leq \mathcal{V}^{\infty} (f_i) + |f_j| \leq \mathcal{V}^{r} (f_i) + |f_j|,\]
where $r < \infty$ is arbitrary.
This difficulty is reflected in the fact that
although having bounded $r$-variation, $r<\infty$, is enough to imply pointwise convergence,
there are functions which converge, but which have unbounded $r$ variation for \emph{any} $r< \infty$. (e.g.\ $\{(-1)^i\frac{1}{\log i+1} \}$)

Despite the increased delicacy of the variation operators, Bourgain proved that for $E_i = [0,i)$, the $r$-variation operators
\[ \mathcal{V}^r(M_if) \]
were of strong-type $(2,2)$ \cite[Corollary 3.26]{B1} for $r > 2$,
\[ \| \mathcal{V}^r(M_if)  \|_{L^2(X)} \leq \frac{C}{r-2} \|f\|_{L^2(X)}\]
for some absolute $C$.
\footnote{It was later shown in \cite{J} that $\V^2(M_if)$ was unbounded on $L^2$; the super-delicacy of the two-variation operator $\V^2$ will be addressed in our context in \S 7.}

In other words, not only do the classical Birkhoff ergodic means $\{M_if\}$ converge in $L^2$, but they do so \emph{rapidly}.

Since Bourgain's celebrated result, establishing variational estimates for families of averaging operators has been the focus of much research in ergodic theory and harmonic analysis (cf. e.g. \cite{J}, \cite{JRW}, or \cite{JSW}). Nevertheless, little research has been directed towards studying variations of averaging operators defined by ``rough,'' arithmetically defined, sets.
This is a natural object of consideration: variational estimates are a strong tool for proving pointwise convergence of averages when a density argument is unavailable.

Indeed, using an easy modification of Bourgain's earliest -- and most straightforward -- proof of the boundedness of maximal function along the squares (i.e.\ $E_i = \{1,2^2, \dots, i^2 \}$) \cite{B0},
\[ \text{there exists an absolute $C$ so that } \ \| \sup_N |M_N f| \|_{L^2(X)} \leq C \|f \|_{L^2(X)}, \]
one can prove (discussed in \S 3) that for each $r > 2, \ \epsilon>0$ there exists an absolute $C_{r,\epsilon}$ so that
\[ \| \mathcal{V}^r (M_N f : N \in \lfloor (1 + \epsilon)^n \rfloor) \|_{L^2(X)} \leq C_{r,\epsilon} \|f \|_{L^2(X)}, \]
which shows that the means $\{ M_N f(x) : N \in \lfloor (1 + \epsilon)^n \rfloor\}$ converge pointwise almost everywhere. Since $\epsilon$ can be taken arbitrarily small, and general means differ from the $(1+\epsilon)$-lacunary means by a multiplicative factor of at most $1+\epsilon$, this result is enough to recover the full pointwise convergence result -- without recourse to Bourgain's difficult ``oscillation'' argument \cite[\S 7]{B0} or his metric-entropy approach \cite[\S 6]{B1}.

Despite the utility of the variational approach, the following problem remains almost untouched:

\begin{problem}\label{bigproblem}
With $(X,\mu,\tau)$ as above, let
\[ P(n) := b_d n^d + \dots + b_1 n + b_0, \ b_j \in \Z, \ b_d > 0 \]
be an integer-valued polynomial, and set $E_i:= \{ P(1),\dots, P(i)\}$.

For which $1<p<\infty$, $2<r<\infty$ do there exist a priori bounds
\[ \| \mathcal{V}^r(M_N f) \|_{L^p(X)} \leq C_{p,r, P} \|f\|_{L^p(X)}? \]
\end{problem}

The main result of this note is a first step towards resolving the above problem. We prove

\begin{theorem}[Polynomial Means Converge Rapidly in $L^2$]\label{main}
Let $r>2$ be arbitrary, and let
\[ P(n) := b_d n^d + \dots + b_1 n + b_0, \ b_j \in \Z, \ b_d > 0 \]
 be an integer-valued polynomial. Then there exists an absolute constant $C_{r,P}$, depending only on $r$ and the polynomial, $P$, so that for any measure-preserving system $(X,\mu,\tau)$, and any $f \in L^2(X)$,
\[ \| \V^r(M_Nf)\|_{L^2(X)} \leq C_{r,P} \|f\|_{L^2(X)}.\]
\end{theorem}

Since variation operators are \emph{semi-local} in the sense of \cite{C}, by Calder\'{o}n's transference principle \cite{C}, Theorem \ref{main} will follow from the result below:
\begin{proposition}\label{main2}
Let $r>2$ be arbitrary, and let
\[ P(n) := b_d n^d + \dots + b_1 n + b_0, \ b_j \in \Z, \ b_d > 0 \]
 be an integer-valued polynomial. Then, with $E_i := \{P(1),P(2),\dots, P(i)\}$ there exists an absolute constant $C_r$ so that for any $f \in l^2(\Z)$,
\[ \| \V^r(K_N*f)\|_{l^2(\Z)} \leq C_{r,P} \|f\|_{l^2(\Z)}.\]
where
\[ K_N*f(x) := \frac{1}{N} \sum_{n \leq N} f(x+ P(n))\]
is the discrete convolution operator.
\end{proposition}

Interpolating this result against Bourgain's celebrated result \cite{B1}
\[ \| \mathcal{V}^\infty (M_N f) \|_{L^p(X)} \leq C_{p,P} \|f\|_{L^p(X)}, \ p>1,\]
yields the following

\begin{proposition}
Suppose $r > \max\{ p, p'\}$. Then for any dynamical system there exist absolute constants $C_{r,p,P}$ so that
\[ \| \V^r(M_Nf) \|_{L^p(X)} \leq C_{r,p,P} \|f\|_{L^p(X)}.\]
\end{proposition}
Again, this result follows from the analogous one on the integer lattice (see \S 7):

\begin{proposition}\label{interp}
For $r > \max\{ p, p'\}$,
\[ \| \V^r(K_N*f) \|_{l^p} \leq C_{r,p,P} \|f\|_{l^p}.\]
\end{proposition}

Finally, we prove that our $L^2$ theory is, in general, sharp. We do so by studying the more delicate $2$-variation operator. Specifically, we prove:

\begin{theorem}[The 2-Variation Operator is Unbounded on $L^2$]
For any $C > 0$, there exists an $f = f_C$ of $L^2$-norm one, but so that
\[ \| \V^2(M_N*f) \|_{L^2(X)} \geq C,\]
where here we fix
\[ M_N*f(x):= \frac{1}{N}\sum_{n \leq N} \tau^{n^2} f(x) ,\]
the (discrete) square means.
\end{theorem}
\begin{remark}
Although this theorem is generalizable to actions associated to other integer-valued polynomials,
for the sake of clarity, we have contented ourselves with the case of the squares.
\end{remark}

This result will follow, by Calder\'{o}n's transference principle \cite{C}, from an analogous version on the torus system
\[(\TT, dx , T: x \mapsto x+ \alpha) \]
for $0 \leq \alpha < 1$; the Kakutani-Rokhlin Lemma \cite[Lemma 4.7]{P} may be then used to transfer the unboundedness of the variation operator to any aperiodic measure-preserving system.

\bigskip

The structure of the paper is as follows:

In $\S 2$ we introduce relevant definitions, and present a few reductions which will be used throughout;

In $\S 3$, we review long variation arguments, and sketch a proof of pointwise convergence along polynomial means;

In $\S 4$, we collect preliminary definitions and lemmas;

In $\S 5$, assuming the number-theoretic Proposition \ref{est}, we prove full variational estimates along the polynomial sequences (i.e. Proposition \ref{main2}) by studying the short variation;

In $\S 6$, we prove Proposition \ref{est};

In $\S 7$, we interpolate our result against Bourgain's theorem to obtain (partial) variational estimates for other $L^p$ spaces; and

In $\S 8$ we prove that our $L^2$-result is sharp -- that in general the $2$-variation associated to our polynomial averages is an unbounded operator.

\subsection{Acknowledgements}
The author would like to thank Lewis Bowen and Akos Magyar for helpful conversations, Michael Lacey for early encouragement, and his advisor, Terence Tao, for his great patience and support.

\subsection{Notation}
For a set $E \subset \Z$, we use $|E|$ to denote $\# E$ the counting measure (cardinality) of the set $E$. We also let $e(t):= e^{2\pi i t}$. For subsets of the torus, $D \subset \TT$, we denote the frequency projection onto $D$ of (finitely supported) functions on the integers  by
\[ f_D(n) := \left( \hat{f} \cdot 1_D \right)^\vee(n). \]

We will make use of the modified Vinogradov notation. We use $X \lesssim Y$, or $Y \gtrsim X$ to denote the estimate $X \leq CY$ for an absolute constant $C$. If we need $C$ to depend on a parameter,
we shall indicate this by subscripts, thus for instance $X \lesssim_p Y$ denotes
the estimate $X \leq C_p Y$ for some $C_p$ depending on $p$. We use $X \approx Y$ as
shorthand for $Y \leq X < 2Y $.

We also make use of big-O notation: we let $O(Y)$ denote a quantity that is $\lesssim Y$, and similarly $O_p(Y)$ a quantity that is $\lesssim_p Y$.

\section{Preliminaries}

For a sequence of functions, $\{f_i\}$, we may divide our study of the \emph{$r$-variation}, $\V^r(f_i)$, into the \emph{long-} and \emph{short-$r$ variation}, respectively defined below:
\[ \aligned
\V^{r,L}(f_i)(x) &:= \sup_{(i_k) \text{ increasing, dyadic}} \left( \sum_k |f_{i_k} - f_{i_{k+1}}|^r \right)^{1/r}(x) \\
\V^{r,S}(f_i)(x) &:= \left( \sum_{n} \left( \sup_{2^n \leq (i_k) \leq 2^{n+1} \text{ increasing }} \sum_k |f_{i_k} - f_{i_{k+1}}|^r \right) \right)^{1/r}(x).
\endaligned \]
Indeed, we have the following easy pointwise inequality:
\begin{lemma}[\cite{JW}, \S 3 ]
\[ \V^r(f_i) \lesssim_r \V^{r,L}(f_i)+\V^{r,S}(f_i).\]
\end{lemma}
\begin{proof}[Sketch]
Fix an increasing sequence of indices $\{i_k\}$. For each pair $(i_k,i_{k+1})$ so that there exists $n(k) < n(k+1)$ with
\[ 2^{n(k) -1 } \leq i_k \leq 2^{n(k)} < 2^{n(k+1)} \leq i_{k+1} \leq 2^{n(k+1)+1},\]
replace
\[ |f_{i_k} - f_{i_{k+1}}| \leq |f_{i_k} - f_{2^{n(k)}}| + |f_{2^{n(k)}} - f_{2^{n(k+1)}}| + |f_{i_{k+1}} - f_{2^{n(k+1)}}|,\]
so that
\[ |f_{i_k} - f_{i_{k+1}}|^r \lesssim_r |f_{i_k} - f_{2^{n(k)}}|^r + |f_{2^{n(k)}} - f_{2^{n(k+1)}}|^r + |f_{i_{k+1}} - f_{2^{n(k+1)}}|^r.\]
Take $l^r$-norms in $k$, and make the above replacements when necessary; the first and third replaced terms become absorbed by the short variation, while each middle term becomes absorbed by the long variation.
\end{proof}

\section{A Sketch of Bourgain's Argument, and the Long Variation}
In this section, we abbreviate the long variation $\V^{r,L}$ by $\V^r$.

We use the argument of \cite{B1}, combined with (an appropriately scaled version of) the variational result of \cite[Proposition 1.2]{BK2} (cf. also \cite[Proposition 4.1]{N1}). The square-function argument of \cite[\S 6]{B1} then is robust enough that one may replace the maximal function along dyadic scales
$\sup_{N \in 2^{\N}} |K_N *f|$ with the long variation $\V^r(K_N*f: N \in 2^{\N})$.

For the sake of clarity, we provide some details of the argument in the case of the squares
\footnote{In fact, in the case of the squares, the argument of \cite[\S\S 3-5]{B0} goes through essentially unchanged. Indeed, the only new observation is that, with $k(x):= \frac{1}{2\sqrt{x}} 1_{[0,1]}$ the quadratic density, not only does the family $\{k_t : t > 0\}$ satisfy the maximal inequality
\[ \| \sup_t |k_t*f| \|_{L^2(\RR)} \lesssim \| f\|_{L^2(\RR)}\]
but also the stronger variational inequality
\[ \| \V^r (k_t*f) \|_{L^2(\RR)} \lesssim \| f\|_{L^2(\RR)}, \]
$r>2$. Working on the spatial side, this follows from Bourgain's \cite[Lemma 3.11]{B1} and convexity; alternatively, one can work on the Fourier side and reduce matters to the more general \cite[Theorem 1.5]{JSW} using Stein's ``universal'' lifting argument \cite[Lemma 11.2.4]{S}.}
\[ K_N*f(x) := \frac{1}{N} \sum_{n \leq N} f(x+n^2).\]

The departure point is that by using techniques from the Hardy-Littlewood circle method, the multipliers
\[ \widehat{K_t}(\alpha) = \frac{1}{t} \sum_{n \leq t} e(-\alpha n^2) \]
can be well-approximated in the $L^2$ sense by a more tractable family of multipliers:
\[
\aligned
\widehat{L_t}(\alpha) &:= \sum_{s\geq 0} \widehat{L_{s,t}}(\alpha) \\
&:= \sum_{s \geq 0} \sum_{a/q \in \R_s} S(a/q) v_t(\alpha - a/q) \phi(10^s(\alpha - a/q)), \endaligned \]
where
\begin{itemize}
\item $1_{[-0.1,0.1]} \leq \phi \leq 1_{[-0.2,0.2]}$ is a smooth cut-off;
\item the sets $\{ \R_s \}$ form an exhaustion of the rationals inside $\TT$:
\[ \R_s := \{ a/q : (a,q) = 1, q \approx 2^s \} \]
(we identify $\R_0 = \{ 0/1 \equiv 1/1 \}$, and recall our convention, $X \approx Y$ means $Y \leq X < 2Y$);
\item the weights
\[ S(a/q) := \frac{1}{q} \sum_{r=1}^q e\left(-\frac{r^2}{q} \right), \]
satisfy $|S(a/q)| \lesssim q^{-\nu}$ for some $\nu > 0$ by Hua's \cite[\S 7, Theorem 10.1]{HUA}; and
\item the $v_t$ are oscillatory ``pseudo-projections''
\[ v_t(\beta) := \int_0^1 e(-\beta t^2 s^2) \ ds, \]
which satisfy the estimates
\[|v_t(\beta) - 1| \lesssim t^2|\beta|, \ \text{ and } |v_t(\beta)| \lesssim \frac{1}{t|\beta|^{1/2}} \]
 by the mean value theorem and van der Corput's estimate on oscillatory integrals.
\end{itemize}

These approximation techniques will reappear in our present context \S\S 4-5; we include a heuristic discussion of Bourgain's (and our) approach.

For each parameter $t$, Bourgain divided the torus into two distinct regions:
\begin{itemize}
\item The $t$-major arcs, which consist of points ``$t$-near'' rationals $a/q$ with ``$t$-small" denominators; and and their complements
\item the $t$-minor arcs.
\end{itemize}

On each $t$-major arc, when $\alpha \sim a/q$ lies near a rational with small denominator, Bourgain showed that
\[ \widehat{K_t}(\alpha) \sim S(a/q) v_t(\alpha - a/q),\]
where $v_t$ is a continuous (integral) analogue of the discrete exponential sum, and the weight $S(a/q)$ measures the lack of uniform distribution of the squares in the residue classes $\mod q$. On each $t$-minor arc, where $\alpha$ is ``$t$-far'' from rational numbers with small denominator, the exponential sum $\widehat{K_t}(\alpha)$ is ``$t$-negligible'' -- $\alpha$ lives too far from any rational $a/q$ which correlates sufficiently quickly to prevent $\widehat{K_t}(\alpha)$ from oscillating itself out of (moral) consideration.

Bourgain was able to quantify these heuristic ideas in his
\begin{lemma}[Lemma 6.14 of \cite{B1}]
There exists some $\nu > 0$ sufficiently small so that
\[ |\widehat{K_t}(\alpha) - \sum_{s\geq 0} \widehat{L_{s,t}}(\alpha) | \lesssim t^{-\nu}.\]
\end{lemma}

For our purposes, for any lacunary constant $\sigma > 1$, and $\|f\|_2 = 1$, we may estimate
\[ \aligned
\| \V^r(K_t*f : t \in \sigma^{\N} ) \|_2 &\leq
\| \V^r(L_t*f : t \in \sigma^{\N} ) \|_2 + \left\| \left( \sum_{k} | K_{\sigma^k}*f - L_{\sigma^k}*f|^2 \right)^{1/2} \right\|_2 \\
&\leq \sum_{s \geq 0}
\| \V^r(L_{s,t}*f : t \in \sigma^{\N} ) \|_2 + \sum_{k\geq 0} (\sigma^{-\nu})^k \\
&\leq \sum_{s \geq 0}
\| \V^r(L_{s,t}*f : t \in \sigma^{\N} ) \|_2 + \frac{1}{1 - \sigma^{-\nu}} . \endaligned \]

Bourgain's next idea was to replace the ``pseudo-projective'' multipliers
\[ \widehat{L_{s,t}}(\alpha) := \sum_{a/q \in \R_s} S(a/q) v_t(\alpha-a/q) \phi(10^s(\alpha - a/q)) \]
with more honestly projective ones
\[ \widehat{L'_{s,t}}(\alpha) := \sum_{a/q \in \R_s} S(a/q) 1_{[-1,1]}(t^2(\alpha-a/q)) \phi(10^s(\alpha - a/q)). \]
Indeed, using our ``pseudo-projective'' estimates on $v_t$, and our bound on $S(a/q)$, we may upper bound each
\[ \aligned
\| \V^r(L_{s,t}*f : t \in \sigma^{\N} ) \|_2 &\leq
\| \V^r(L'_{s,t}*f : t \in \sigma^{\N} ) \|_2 + \left\| \left( \sum_{k} | L_{\sigma^k}*f - L'_{\sigma^k}*f|^2 \right)^{1/2} \right\|_2 \\
&\lesssim \| \V^r(L'_{s,t}*f : t \in \sigma^{\N} ) \|_2 + 2^{-s \nu} \cdot \frac{1}{\sigma - 1}, \endaligned \]
and thus
\[
\sum_{s \geq 0}
\| \V^r(L_{s,t}*f : t \in \sigma^{\N} ) \|_2
\lesssim
\sum_{s \geq 0}
\| \V^r(L'_{s,t}*f : t \in \sigma^{\N} ) \|_2 + \frac{1}{\sigma -1 }.
\]
Define now the multiplier
\[ \widehat{ B_{s,t} g } := \sum_{a/q \in \R_s} 1_{[-1,1]}(t^2(\alpha - a/q)) \hat{g}( \alpha ).\]
With 
\[ \hat{g_s}(\alpha) := \hat{f}(\alpha) \cdot \sum_{a/q \in \R_s} S(a/q) \phi(10^s(\alpha - a/q))\]
defined via the Fourier transform, so
\[ \|g_s\|_2 \leq 2^{-s\nu} ,\]
 we have
\[ \widehat{ B_{s,t} g_s } \equiv \widehat{ L'_{s,t} f }.\]
Now, suppose we knew that for each $g$,
\[ (*) \; \; \; \; \; \;  \| \V^r( \widehat{ B_{s,t} g } : t \in \sigma^{\N} ) \|_2 \lesssim \left( \frac{r}{r-2} \right)^2 s^2 \cdot \frac{1}{\sigma -1} \|g\|_2;\]
then we could conclude
\[ \aligned
\sum_{s \geq 0}
\| \V^r(L'_{s,t}*f : t \in \sigma^{\N} ) \|_2 &\equiv \sum_{s \geq 0}
\| \V^r(B_{s,t}*g_s : t \in \sigma^{\N} ) \|_2 \\
&\lesssim \frac{1}{\sigma -1} \cdot \left( \frac{r}{r-2} \right)^2 \cdot \sum_{s\geq 0} s^2 \| g_s\|_2 \\
&\lesssim \frac{1}{\sigma -1} \cdot \left( \frac{r}{r-2} \right)^2 \sum_{s \geq 0} s^2 2^{-s\nu} \\
&\lesssim \frac{1}{\sigma -1} \cdot \left( \frac{r}{r-2} \right)^2, \endaligned\]
from which it follows that
\[ \| \V^r(K_t*f : t\in \sigma^{\N} ) \|_2 \lesssim \left( \left(\frac{r}{r-2} \right)^2 \cdot  \frac{1}{\sigma - 1} \cdot \frac{1}{1 - \sigma^{-\nu} } \right).\]

It remains to prove the following (slightly more general) proposition, in the spirit of \cite[\S 4]{B1}:

\begin{proposition}\label{entropy}
Suppose $\lambda_1 < \dots < \lambda_N \subset \TT$ are $\tau$-separated frequencies: $|\lambda_i - \lambda_j| > \tau$ for $i \neq j$.
For $k$ so that $\sigma^{-k} < \frac{1}{100} \tau$ (say) let
\[ R_k:= \{ \alpha \in \TT : |\alpha - \lambda_j| \leq \sigma^{-k} \},\]
and define
\[ \Delta^{\TT}_kf(n):= \left( 1_{R_k} \hat{f} \right)^{\vee}(n).\]
Then
\[ \| \V^r( \Delta^{\TT}_k f ) \|_{l^2} \lesssim \left( \frac{r}{r-2} \log N \right)^2 \cdot \frac{1}{\sigma-1} \|f\|_{l^2}.\]
\end{proposition}

In particular, taking $\{ \lambda_1,\dots, \lambda_N \}$ to be the $\lesssim 4^s$ elements of $\R_s$, and $R_k$ the pertaining $\sigma^{-k}$-neighborhood, we find that
\[ \aligned 
\| \V^r(B_{s,t}*g_s : t \in \sigma^{\N} ) \|_2 &\lesssim \frac{1}{\sigma -1} \cdot \left( \frac{r}{r-2} \right)^2 \left( \log | \R_s | \right)^2 \|g_s \|_2 \\
&\lesssim
\frac{1}{\sigma -1} \cdot \left( \frac{r}{r-2} \right)^2 s^2 \|g_s \|_2, \endaligned\]
proving $(*)$. Indeed, upon establishing Proposition \ref{entropy}, we will
have proven the following:
\begin{proposition}\label{lac}
With $K_N*f(x):= \frac{1}{N} \sum_{n \leq N} f(x +P(n))$ as above, for any $r>2$
\[ \| \V^r (K_N*f: N \in 2^{\N})  \|_{l^2} \lesssim_{r,P} \|f\|_{l^2}.\]
\end{proposition}

And moreover

\begin{cor}
For $r > 2$, $\sigma>1$, maintaining the above notation,
\[ \| \mathcal{V}^r (K_N*f : N \in \sigma^{\N} ) \|_{L^2(X)} \lesssim_{r,P} \left( \frac{1}{\sigma -1} \cdot \frac{1}{1 - \sigma^{-\nu}} \right) \|f \|_{L^2(X)}. \]
In particular, the means $\{ K_N*f(x) : N \in \lfloor \sigma^n \rfloor\}$ converge pointwise almost everywhere.
\end{cor}
Since $\sigma$ can be taken arbitrarily close to one, and general means differ from the $\sigma$-lacunary means by a multiplicative factor of at most $\sigma$, this result is enough to recover the $L^2$-version of the following result:
\begin{theorem}[\cite{B1}, Theorem 5]
For any measure-preserving system $(X,\Sigma,\mu,\tau)$ and any function $f \in L^2(X)$, the means
\[ \left\{ \frac{1}{N} \sum_{n \leq N} \tau^{P(n)} f(x) \right\} \]
converge pointwise $\mu$-a.e.
\end{theorem}
\begin{remark}
The following theorem in fact holds for $f \in L^p(X), \ p >1$: provided that $f$ has ``mild enough'' singularities, polynomial means converge $\mu$-a.e. This result was later proven to be sharp \cite{BM}.
\end{remark}

We establish Proposition \ref{entropy} in the following subsection.

\subsection{Proof of Proposition \ref{entropy}}
By arguing as in \cite[Lemma 4.4]{B1} (see also \cite[Lemma 5]{T}), it is enough to prove the analogous result on $\RR$, where we may take advantage of the dilation structure, and prove the analogous result under the hypothesis that the frequencies are $1$-separated.

\begin{lemma}\label{R}
Suppose $\xi_1 < \dots < \xi_N \subset \RR$ are $1$-separated, and similarly define for $k$ so large that $\sigma^{-k} \leq \frac{1}{100}$
\[ \Delta_kf(x):= \left( 1_{R_k} \hat{f} \right)^{\vee}(x),\]
where now $R_k:= \{ \xi : |\xi - \xi_j| \leq \sigma^{-k} \text{ for some j} \}$.
Then
\[ \| \V^r( \Delta_k f ) \|_2 \lesssim \left( \frac{r}{r-2} \log N \right)^2 \cdot \frac{1}{\sigma-1} \|f\|_2.\]
\end{lemma}

\begin{proof}[Proof of Lemma \ref{R}]
We begin with a reduction:

Let $1_{[-1/2,1/2]} \leq \hat{\phi} \leq 1_{[-1,1]}$ be a smooth function,
\[ \phi_k := \frac{1}{\sigma^k} \phi(\frac{x}{\sigma^{k}}),\]
and define the operators
\[ \aligned
D_k f &:= \sum_{j=1}^N \int \hat{f}(\xi) \hat{\phi_k}(\xi - \xi_j) e(\xi x) \ d\xi \\
&=
\sum_{j=1}^N e(\xi_j x) \int \hat{f}(\xi+ \xi_j) \hat{\phi_k}(\xi) e(\xi x) \ d\xi. \endaligned\]

We may majorize
\[ \V^r(\Delta_kf) \leq \V^r(D_kf) + \left( \sum_k |\Delta_kf - D_k f|^2 \right)^{1/2},\]
and using the separation hypothesis, we see that
\[ \|  \left( \sum_k |\Delta_kf - D_k f|^2 \right)^{1/2} \|_2 \lesssim \frac{1}{\sigma-1},\]
since
\[ \sup_\xi \sum_k |1_{R_k}(\xi) - \sum_{j=1}^N \hat{\phi}(\xi-\xi_j) | \lesssim \frac{1}{\sigma-1} \]
by sparsification.

In particular, we have
\[ \| \V^r(\Delta_kf) \|_2 \leq \| \V^r(D_kf) \|_2 + \frac{1}{\sigma-1}\|f\|_2.\]

But the arguments of \cite[\S 2]{BK2} show that
\[ \| \V^r(D_kf) \|_2 \lesssim \left( \frac{r}{r-2} \log N \right)^2 \|f\|_2,\]
which proves the lemma, and with it, Proposition \ref{entropy}.
\end{proof}

This concludes our treatment of the Long Variation result in the case of the squares.

The case of general polynomial averages follows a similar argument; the only modifications are to the definitions of the weights and the ``pseudo-projections.'' See \cite[\S \S 5-6]{B1} or \S 4 below for the appropriate generalizations.

\section{The Short Variation}
In this section, we set the stage for a proof that $\V^{r,S}$, and therefore $\V^r$ itself (see above), is $L^2$-bounded. Henceforth, we abbreviate the short variation $\V^{r,S}$ by $\V^r$;

\subsection{Fourier Preliminaries}
This is an $L^2$-problem, so we will make use of the Fourier transform.
We begin with some notation; wherever possible, we will maintain that of Bourgain.

Throughout, we shall regard our polynomial
\[ P(n) := b_d n^d + \dots b_1 n + b_0 \]
as fixed, and $1 \gg \delta >0$ will be a small but fixed constant.

We recall the exhaustion of the rationals inside $\TT$
\[ \bigcup_s \R_s := \bigcup_s \left\{ \frac{a}{q} : (a,q)=1, q \approx 2^s \right\},\]
where we use $X \approx Y$ to mean $Y \leq X \leq 2Y$.

For $t \approx 2^n$, we define the \emph{$n$-major arc}
\[ \aligned
\MM_n 
&:= \bigcup_{s \leq n\delta} \bigcup_{a/q \in \R_s } \MM_n (a/q) \\
&:= \bigcup_{a/q \in \R_s, s \leq n \delta} \left\{ \alpha : |\{ b_d \alpha \} - a/q | < 2^{-n(d-\delta)} \right\}, \endaligned \]
where we use $\{ x\} := x - \lfloor x \rfloor$ to denote the fractional part.
We continue to identify $0 \equiv 1$, so $\R_0 := \{ \frac{0}{1} \}$ and
\[ \MM_n(0/1) := \{ \alpha : \| \{ b_d \alpha \} \|_{\TT} < 2^{-n (d-\delta) } \},\]
where we use $\| x \|_{\TT}$ to denote the distance to the nearest integer.

For $0 \leq i < b_d$, \footnote{For technical ease, on first reading we recommend assuming that $b_d=1$, i.e.\ that $P$ is a monic polynomial.}
we further define
\[ \MM_n^i(a/q) := \left\{ \alpha \in \left[\frac{i}{b_d},\frac{i+1}{b_d} \right) \right\} \cap \MM_n(a/q) \]
to be intersection of each major arc with each of the $b_d$ distinct ``pre-intervals''
\[ \left\{ |\beta - a/q| \leq 2^{-n(d-\delta)} \right\}\]
under the map $\alpha \mapsto b_d \alpha$.
(Throughout, we assume that $n \gg_P 1$ is sufficiently large that there are $b_d$ distinct such pre-intervals.)

We simply define the minor arcs
\[ \mm_n := \TT \smallsetminus \MM_n.\]

With
\[ s, s' \leq n\delta, \ a/q \in \R_s, \ b/r \in \R_{s'},\]
 we remark that any two pre-intervals corresponding to $a/q$, $b/r$ are distinct. For otherwise we would have $\alpha$ such that
\[ | \{ b_d \alpha \} - a/q | + | \{ b_d \alpha \} - b/r | \lesssim 2^{-n (d- \delta)},\]
while
\[ | a/q - b/r| \geq 1/qr \gtrsim 2^{-2n\delta},\]
for the desired contradiction, since $\delta$ is sufficiently small.


Now, for each $a/q \in \R_s, \ s \leq n\delta$, we define
\[ q_i = q_i(a/q, P,i), \ 0 \leq i < b_d \]
to be the least common denominator of
\[ a/q, \ b_{d-1}/ b_d \cdot (a/q + i) , \ \dots, \ b_1/b_d \cdot (a/q + i) \]
i.e.\ $q_i$ is as small as possible such that there exist integers
\[ a^i_d= a^i_d(a/q, P,i), \ a^i_{d-1}= a^i_{d-1}(a/q, P,i), \ \dots, \ a^i_1= a^i_1(a/q, P,i)\]
satisfying
\[ \aligned
\big( a/q, b_{d-1}/ b_d \cdot (a/q +i) , \dots, b_1/b_d \cdot (a/q+i) \big) &= \big(a^i_d/q_i, a^i_{d-1}/q_i , \dots, a^i_1/q_i \big) \text{ and } \\
(a^i_d,a^i_{d-1},\dots,a^i_1, q_i) &= 1. \endaligned\]
We remark that for each $i$, $q | q_i$, and that $q_i \leq b_d q \lesssim_P q$.

With the above notation in mind, we define the weight
\[ S_P^i (a/q) := \frac{1}{q_i} \sum_{r=1}^{q_i} e \left( \frac{-a^i_d r^d - \dots - a^i_1 r}{q_i} \right).\]
By \cite[\S 7, Theorem 10.1]{HUA} for any $\epsilon > 0$ we have the estimate
\[ |S_P^i(a/q)| \lesssim_\epsilon (q_i)^{\epsilon - 1/d} \lesssim q^{\epsilon - 1/d};\]
all we need is that there exist small $\nu > 0$ (determined below) so that
\[ |S_P^i(a/q) | \lesssim q^{-\nu}.\]
We also define the oscillatory ``pseudo-projections''
\[ v_t(\beta) := \int_0^1 e( -\beta t^d s^d) \ ds;\]
for $t \approx 2^n$ we have
\[ \aligned
|v_t(\beta) - 1| &\lesssim 2^{nd} |\beta|, \\
|v_t(\beta)| &\lesssim \frac{1}{2^n|\beta|^{1/d}}, \endaligned\]
by the mean-value theorem and van der Corput's lemma on oscillatory integrals.

For $t \approx 2^n$, we define the multipliers
\[ \widehat{C_t}(\alpha) = \widehat{K_t}(\alpha) - \widehat{K_{2^n}}(\alpha).\]
The $L^2$-smoothness of the family of multipliers
\[ t \mapsto \widehat{K_t}(\alpha) \]
is captured by the following proposition, whose proof we defer to \S 6 below.

\begin{proposition}\label{est}
Suppose $\frac{a}{q} \in \R_s, s \leq n\delta$, and that $t \approx 2^n$. Then there exists $\nu >0$ so that
\begin{enumerate}
\item For any $\alpha \in \TT$, $|\widehat{C_t - C_{t+1}}(\alpha)| \equiv |\widehat{K_t - K_{t+1}}(\alpha)| \lesssim 2^{-n}$;
\item For $\alpha \in \mm_n$ a minor arc, $|\widehat{C_t}(\alpha)| \lesssim 2^{-n \nu}$; and
\item For $\alpha \in \MM_{n}(\frac{a}{q})$ a major arc
\[ |\widehat{C_t}(\alpha)| \lesssim 2^{-\nu s} \left( \min \left\{ 2^n|\{ b_d\alpha\}- a/q|^{1/d}, \frac{1}{2^n|\{ b_d\alpha\} - a/q|^{1/d}} \right\} + 2^{-n/2} \right).\]
\end{enumerate}
\end{proposition}

We also include the following elementary lemma. This result is essentially due to Bourgain (cf. \cite[Lemma 3.11]{B1}, and \cite[Proposition 2.12]{J} as well), and we will make repeated use of it.
\begin{lemma}\label{smooth}
Suppose that $\{B_n\}_{n=1}^N$ are a family of operators which act on (finitely supported) $l^2$ functions by multiplication on the fourier side
\[ \widehat{B_nf}(\beta) := m_n(\beta) \hat{f}(\beta).\]
Suppose that for each $1 \leq n \leq N$
\[ \aligned
&\sup_{\beta \in \supp \hat{f}} |m_n(\beta)| \leq A \text{ and } \\
&\sup_{\beta \in \supp \hat{f}} |m_n(\beta) - m_{n+1}(\beta)| \leq a. \endaligned \]
Then $\| \V^2(B_nf) \|_{l^2} \lesssim \sqrt{NAa} \|f\|_{l^2}$.
\end{lemma}
\begin{proof}
With $L$ a positive integer to be determined, and $1= a_1 < a_2 < \dots a_L = N$ (almost) equally spaced indices (so $a_{i+1} - a_{i} \approx \frac{N}{L}$), we may pointwise dominate
\[ \aligned
\V^2( B_nf) &\leq \left(\sum_{i=1}^L |B_{a_i}f|^2 \right)^{1/2} +
\left( \sum_{i=1}^L \left( \V^2( B_nf : a_i \leq n \leq a_{i+1}) \right)^2 \right)^{1/2} \\
&\leq \left(\sum_{i=1}^L |B_{a_i}f|^2 \right)^{1/2} +
\left( \sum_{i=1}^L \left( \sum_{n=a_i}^{a_{i+1}} |B_nf - B_{n+1}f| \right)^2 \right)^{1/2} \\
&\lesssim \left(\sum_{i=1}^L |B_{a_i}f|^2 \right)^{1/2} + \left(\frac{N}{L} \right)^{1/2} \left( \sum_{n=1}^N |B_nf - B_{n+1}f|^2 \right)^{1/2}, \endaligned\]
where we used Cauchy-Schwartz in the last inequality.

We take $l^2$-norms, and use Plancherel's theorem to majorize the first summand
\[ \aligned
\left\| \left(\sum_{i=1}^L |B_{a_i}f|^2 \right)^{1/2} \right\|_{l^2} &= \left( \sum_{i=1}^L \left\| m_{a_i} (\beta) \hat{f}(\beta) \right\|_{L^2}^2 \right)^{1/2} \\
&\leq \left( \sum_{i=1}^L A^2 \left\| \hat{f} \right\|_{L^2}^2 \right)^{1/2} \\
&\leq \sqrt{L} A \|f\|_{l^2} \endaligned \]
and the second summand
\[ \aligned
\left\| \left(\frac{N}{L} \right)^{1/2} \left( \sum_{n=1}^N |B_nf - B_{n+1}f|^2 \right)^{1/2} \right\|_{l^2} &=
\left(\frac{N}{L} \right)^{1/2} \left( \sum_{n=1}^N \left\|(m_n - m_{n+1})(\beta) \hat{f}(\beta) \right\|_{L^2}^2 \right)^{1/2} \\
&\leq \left(\frac{N}{L} \right)^{1/2}  \sqrt{N}a \|f\|_{l^2} \\
&= \frac{Na}{ \sqrt{L}} \|f\|_{l^2}. \endaligned\]
Setting $L \approx \frac{Na}{A}$ yields the result.
\end{proof}

\section{The Proof of Proposition \ref{main2}}
By our long variation result, it suffices to prove
\[ \left( \sum_n \| \V^2_n (K_t*f) \|_2^2 \right)^{1/2} \lesssim \|f\|_2,\]
where we abbreviate
\[\V^2_n (K_t*f) := \V^2 (K_t*f : t \approx 2^n).\]

We split $f = f_{\mm_n} + f_{\MM_n}$ as a projection onto $n$-minor and $n$-major arcs, and majorize
\[ \left( \sum_n \| \V^2_n (K_t*f) \|_2^2 \right)^{1/2} \leq \left( \sum_n \| \V^2_n (K_t*f_{\mm_n}) \|_2^2 \right)^{1/2} + \left( \sum_n \| \V^2_n (K_t*f_{\MM_n}) \|_2^2 \right)^{1/2}.\]

We use Proposition \ref{est} and Lemma \ref{smooth} to control the first summand. Specifically, in the notation of Lemma \ref{smooth} we may take
\[ N = 2^n, A = 2^{-n\nu}, \text{ and } a = 2^{-n},\]
so that we may estimate
\[
\left( \sum_n 2^{-n \nu} \| f_{\mm_n} \|_2^2 \right)^{1/2} \leq
\left( \sum_n 2^{-n \nu} \| f \|_2^2 \right)^{1/2} \leq \| f\|_2.\]

We therefore restrict our attention to the second term in the above summand; we will prove the following
\begin{proposition}\label{mainprop}
In the above notation,
\[ \sum_n \| \V^2_n (K_t*f_{\MM_n}) \|_2^2  \lesssim \|f\|_2^2.\]
\end{proposition}
We make a few remarks before we turn to the proof proper:

For $t \approx 2^n$, \[ \V^2_n (K_t*f_{\MM_n}) \equiv \V^2_n (C_t*f_{\MM_n}),\]
since we are summing over \emph{differences} of operators.

Since, roughly speaking, for $t \approx 2^n$, on $\MM_n^i(\frac{a}{q})$ we have
\[ \aligned
|\widehat{C_t}(\alpha)| ``&=" |S_P^i(\frac{a}{q})| \left|v_t(\{b_d\alpha\} - \frac{a}{q}) - v_{2^n}(\{b_d\alpha\}-\frac{a}{q}) \right|
\\
&\lesssim q^{-\nu} \min\left\{2^{dn}|\{b_d\alpha\} -\frac{a}{q}|,\frac{1}{2^{n}|\{b_d\alpha\} -\frac{a}{q}|^{1/d}} \right\}, \endaligned\]
it makes sense to partition the major arcs according to both the size of the denominators of, and the distance to, our rationals $a/q$.
We therefore further decompose our major arcs:

For $k \gg 2s$, we introduce
\[ \MM_n = \bigcup_{s \leq n \delta} \bigcup_{l \geq -n\delta} \RRR_{s,nd + l},\]
where
\[ \RRR_{s,k} := \bigcup_{a/q \in \R_s} \left\{ \alpha : | \{ b_d \alpha \} - a/q | \approx 2^{-k} \right\}.\]

\begin{proof}
On each $\RRR_{s,nd +l}$ we bound
\[ |\widehat{C_t}(\alpha)| \lesssim 2^{-\nu s} \left( 2^{-|l|/d} + 2^{-n/2} \right).\]

If we define the critical $l_n := \frac{nd}{2}$ to be the unique distance where $2^{-|l|/d} = 2^{-n/2}$, we have
\[
|\widehat{C_t}(\alpha)| \lesssim
\begin{cases} 2^{-\nu s}  2^{-|l|/d} &\mbox{if } l \leq l_n \\
2^{-\nu s}  2^{-l_n/d} \equiv
2^{-\nu s}  2^{-n/2} & \mbox{if } l > l_n \end{cases}.\]

We now collect $\bigcup_{l> l_n} \RRR_{s,nd+l} =: \RRR_{s,n^*}$, and use Lemma \ref{smooth} to majorize
\[ \aligned
&\| \V_n^2(C_t*f_{\MM_n}) \|_2^2 \\
& \qquad \leq \left( \sum_{s \leq n \delta} \sum_{l= - n \nu}^{l_n} \| \V_n^2(C_t*f_{\RRR_{s,nd+l}}) \|_2 + \| \V_n^2(C_t*f_{\RRR_{s,n^*}}) \|_2
\right)^2 \\
& \qquad \lesssim
\left( \sum_{s \leq n \delta} \sum_{l= - n \nu}^{l_n} 2^{-s\nu/2} 2^{-|l|/2d} \| f_{\RRR_{s,nd+l}} \|_2 + 2^{-s\nu/2}2^{-l_n/2d} \| f_{\RRR_{s,n^*}} \|_2 \right)^2 \\
& \qquad \lesssim
\sum_{s \leq n \delta} 2^{-s\nu/2} \sum_{l= - n \nu}^{l_n} 2^{-|l|/2d} \| f_{\RRR_{s,nd+l}} \|_2^2 + 2^{-l_n/2d} \| f_{\RRR_{s,n^*}} \|_2^2, \endaligned\]
where $f_{\RRR_{s,k}}$ denotes the fourier projection onto $\RRR_{s,k}$, etc.\ and we used Cauchy-Schwarz in $\{ s \leq n \nu, l \leq l_n \}$ in the final inequality.

Summing the foregoing over $n$, and interchanging the $(n,s)$-order of summation yields the upper estimate
\[
\sum_n \| \V^2_n (K_t*f_{\MM_n}) \|_2^2 \leq
\sum_{s } 2^{-s\nu/2} \left( \sum_{n= \frac{s}{\delta}} \sum_{l= - n \nu}^{l_n} 2^{-|l|/2d} \| f_{\RRR_{s,nd+l}} \|_2^2 + 2^{- n/4} \| f_{\RRR_{s,n^*}} \|_2^2 \right);\]
we will show that each bracketed term is $\lesssim \|f\|_2^2$.

To do so, with $s$ fixed, we expand the bracketed expression, make the change of variables $k = nd +l$, and interchange $(n,k)$ order of summation to obtain
\[ \aligned
& \sum_{n= \frac{s}{\delta}} \sum_{l= - n \nu}^{l_n} 2^{-|l|/2d} \| f_{\RRR_{s,nd+l}} \|_2^2 + 2^{- n/4} \| f_{\RRR_{s,n^*}} \|_2^2 \\
& \qquad =
\sum_{n= \frac{s}{\delta}} \sum_{l= - n \nu}^{l_n} 2^{-|l|/2d} \| f_{\RRR_{s,nd+l}} \|_2^2 + 2^{- n/4} \sum_{l>l_n} \| f_{\RRR_{s,nd + l}} \|_2^2 \\
& \qquad =
\sum_{n= \frac{s}{\delta}} \sum_{k = nd - n \nu}^{\frac{3nd}{2}  } 2^{-|k-nd|/2d} \|f_{\RRR_{s,k}}\|_2^2 + 2^{- n/4} \sum_{k>\frac{3nd}{2}} \| f_{\RRR_{s,k}} \|_2^2 \\
& \qquad =
\sum_{k = \frac{s}{\delta} (d - \nu)} \sum_{ \max\{ \frac{s}{\delta}, \frac{k}{3d/2} \} }^{\frac{k}{d-\nu}} 2^{-|k-nd|/2d} \|f_{\RRR_{s,k}}\|_2^2 +
\sum_{k = \frac{s}{\delta} 3d/2 } \sum_{ \frac{s}{\delta}}^{\frac{k}{3d/2}} 2^{-n/4} \|f_{\RRR_{s,k}}\|_2^2 \\
&\qquad \lesssim
\sum_{k = \frac{s}{\delta} (d - \nu)} \|f_{\RRR_{s,k}}\|_2^2 + \sum_{k = \frac{s}{\delta} 3d/2 } \|f_{\RRR_{s,k}}\|_2^2 \\
& \qquad \lesssim \|f\|_2^2, \endaligned\]
as desired.
\end{proof}

\section{Proof of Proposition \ref{est}}
In this section we prove Proposition \ref{est}, and thereby conclude the argument.

\begin{proposition}
Suppose $\frac{a}{q} \in \R_s, s \leq n\delta$, and that $t \approx 2^n$. Then there exists $\nu >0$ so that
\begin{enumerate}
\item For any $\alpha \in \TT$, $|\widehat{C_t - C_{t+1}}(\alpha)| \equiv |\widehat{K_t - K_{t+1}}(\alpha)| \lesssim 2^{-n}$;
\item For $\alpha \in \mm_n$ a minor arc, $|\widehat{C_t}(\alpha)| \lesssim 2^{-n \nu}$; and
\item For $\alpha \in \MM^i_{n}(\frac{a}{q})$ each segment of a major arc
\[ |\widehat{C_t}(\alpha)| \lesssim 2^{-\nu s} \left( \min \left\{ 2^n|\{ b_d\alpha\}- a/q|^{1/d}, \frac{1}{2^n|\{ b_d\alpha\} - a/q|^{1/d}} \right\} + 2^{-n/2} \right).\]
\end{enumerate}
\end{proposition}

\begin{proof}
The first point follows trivially from the triangle inequality, so we begin the proof proper with the minor arcs.

Here, whenever $|\{b_d\alpha\} - a/q | < 2^{-n(d-\delta)}$ we necessarily have $q \gtrsim 2^{n \delta}$ (the approximate inequality comes from the fact that we admit all $a/q \in \R_s, \ s \leq n\delta$ into the definition of our major arcs, rather than simply $a/q$ such that $q < 2^{n\delta}$). By Dirichlet's theorem, we may choose a reduced fraction
\[ x/y, \ (x,y) = 1, \ y \leq 2^{n(d-\delta)}\]
 so that
\[ |\{ b_d\alpha\} - x/y | \leq \frac{1}{y 2^{n(d-\delta)}} \leq \frac{1}{y^2}.\]
By Weyl's inequality \cite[Lemma 2.1]{VA}, we have
\[ \left| \frac{1}{t} \sum_{n=1}^t e(- b_d\alpha n^d - \dots - b_1 \alpha n) \right| = |\widehat{K_t(\alpha)}| \lesssim_\epsilon
t^\epsilon \left( 1/y+ 1/t+y/t^d \right)^{1/2^{d-1}},\]
which leads to the effective estimate $|\widehat{K_t}(\alpha)| \lesssim 2^{-n\nu}$ for some $\nu = \nu(\epsilon, \delta) > 0$. The triangle inequality yields the second point.

We next turn to the major arcs.

Suppose $\alpha \in \MM^i_n(a/q)$, so that we may express
\[ b_d \alpha = i +a/q + \beta, \ \text{ where } |\beta| < 2^{-n(d-\delta)}, \ q \lesssim 2^{n\delta}.\]
With $q_i$ as in \S 4 above, and $m \leq t \approx 2^n$, we express $m = p q_i + r$ and write
\[ \aligned
P(m) \alpha &= (b_d\alpha) (pq_i +r)^d + \frac{b_{d-1}}{b_d} (b_d \alpha)(pq_i+r)^{d-1} + \dots + \frac{b_{1}}{b_d} (b_d \alpha)(pq_i+r) \\
&= (i + a/q + \beta) (pq_i +r)^d + \frac{b_{d-1}}{b_d} (i + a/q + \beta) (pq_i+r)^{d-1} + \dots + \frac{b_{1}}{b_d} (i + a/q + \beta) (pq_i+r) \\
&\equiv \frac{a^i_{d}}{q_i}(pq_i + r)^d + \beta(pq_i)^d +  \frac{a^i_{d-1}}{q_i}(pq_i + r)^{d-1} + \dots + \frac{a^i_{1}}{q_i}(pq_i + r) + O_P(2^{n (2 \delta -1)}) \mod 1 \\
&\equiv \beta (pq_i)^d + \left( \frac{a_d^i r^d + \dots + a_1^i r}{q_i} \right) + O_P(2^{n (2 \delta -1)}) \mod 1, \endaligned\]
so that for $m = pq_i+r$,
\[ e(-P(m) \alpha) = e\left( - \frac{a^i_d r^d +  \dots + a^i_1 r}{q_i} \right) e( -\beta (pq_i)^d) + O_P(2^{n(2\delta - 1)}).\]
Now, with $t = p_t q_i + r_t$, we have
\[ \aligned
\frac{1}{t} \sum_{m \leq t} e(-P(m) \alpha) &= \frac{1}{t} \sum_{m \leq p_t q_i} e(-P(m) \alpha) + O_P(2^{n(\delta - 1)}) \\
&= \frac{q_i}{t} \sum_{p=0}^{p_t - 1} e(-\beta (pq_i)^d) \cdot \frac{1}{q_i} \sum_{r=1}^{q_i} e\left( - \frac{a_d^ir^d + \dots + a^i_1 r}{q_i} \right) + O_P(2^{n(2\delta - 1)}) \\
&= \frac{q_i}{t} \int_0^{p_t} e( - \beta (q_i)^d s^d ) \ ds \cdot S^i_P(a/q) + O_P(2^{n(2\delta - 1)}) \\
&= \frac{q_i}{t} \int_0^{q_i/t}e( - \beta (q_i)^d s^d ) \ ds \cdot S^i_P(a/q) + O_P(2^{n(2\delta - 1)}) \\
&= \int_0^1 e( - \beta t^d s^d) \ ds \cdot S^i_P(a/q) + O_P(2^{n(2\delta - 1)}) \\
& = v_t(\beta) S^i_P(a/q) + O_P(2^{n(2\delta - 1)}), \endaligned\]
where we used that for $|p-s| \leq 1$
\[ e(-\beta (q_i)^d \cdot p^d ) = e(-\beta (q_i)^d \cdot s^d ) + O_P(2^{n(2\delta -1)})\]
in passing to the third line.
The upshot is that on $\MM^i_N(a/q)$ we have
\[ \widehat{C_t}(\alpha) = S^i_P(a/q) \cdot \left( v_t( \{b_d \alpha \} - a/q ) - v_{2^n}( \{b_d \alpha \} - a/q ) \right) + O_P(2^{n(2\delta -1)}).\]
Taking into account our estimates on $S^i_P(a/q), \ v_t$ we have the upper bound
\[ \aligned
|\widehat{C_t}(\alpha)| &\lesssim 2^{-s \nu} \left( \min \left\{ 2^{nd}| \{b_d \alpha \} - a/q | , \frac{1}{2^n| \{b_d \alpha \} - a/q |^{1/d}} \right\} \right) + O_P(2^{n(2\delta -1)}) \\
&\lesssim 2^{-s \nu} \left( \min \left\{ 2^{n}| \{b_d \alpha \} - a/q |^{1/d} , \frac{1}{2^n| \{b_d \alpha \} - a/q |^{1/d}} \right\} + 2^{-n/2} \right),
 \endaligned\]
since $2^{s\nu} \leq 2^{n\delta \nu}$ for $s \leq n \delta$.
\end{proof}

\section{Interpolation}
In this section, we will interpolate our $L^2$-based $\{ \V^r \}_{r > 2}$ estimates against Bourgain's $L^p, \ p>1$-based $\V^\infty$ estimate to some partial variational estimates in other $L^p$ spaces, and thereby make further progress towards understanding Problem \ref{bigproblem}.

We begin with the following general mixed-norm (complex) interpolation lemma, whose proof follows the same lines as the classical Riesz-Thorin interpolation theorem.


\begin{lemma}\label{absint}
Suppose that $T$ is a linear operator, bounded
\[ \aligned
T &: l^{p_0}_x \to l^{p_0}_x ( l^{r_0}_k) \\
T &: l^{p_1}_x \to l^{p_1}_x ( l^{r_1}_k). \endaligned \]
If
$\frac{1}{p} = \frac{1-\theta}{p_0} + \frac{\theta}{p_1}$ and
$\frac{1}{r} = \frac{1-\theta}{r_0} + \frac{\theta}{r_1}$ are determined by convexity, $T$ is bounded
\[ T: l^{p}_x \to l^{p}_x ( l^{r}_k) \]
as well.
\end{lemma}

To use this lemma, we remark that we may linearize the variation, by setting
\[ Tf(x,k):= (K_{N_k(x)} - K_{N_{k+1}(x)})*f(x),\]
where $\{N_k\}$ are any (finite) collection of measurable functions. We are ready to prove our

\begin{proposition}
Suppose $r > \max\{ p, p'\}$. Then
\[ \| \V^r(K_N*f) \|_{l^p} \lesssim_{p,r} \|f\|_{l^p}.\]
\end{proposition}
By Bourgain's Theorem and the $L^2$ variational result, we know that
\[ \| \V^r(K_N^*f) \|_2 \lesssim_r \|f\|_2 \]
for $r>2$,
and that
\[ \| \V^\infty (K_N^*f) \|_p \lesssim_p \|f\|_p \]
for $p >1$.

For each $p$ fixed, the task is to choose $r$ as small as possible so that $\theta$ satisfies
\[ \frac{1}{p} = \frac{1-\theta}{p_0} + \frac{\theta}{2} \]
and
\[ \frac{1}{r} = \frac{1 - \theta}{ \infty} + \frac{\theta}{r_1} = \frac{\theta}{r_1},\]
where $p_0>1$ and $r_1 > 2$. To minimize $r$, we want to take $r_1$ close to $2$, and $\theta$ as large as possible.
Consequently, it is in our interest to choose $p_0$ as far from $p$ as possible (i.e.\ $p_0= \infty$ for $p >2$, and $p_0$ near $1$ for $p<2$).

With this strategy in mind, we turn to the:
\begin{proof}
We begin with the slightly more involved case, $p < 2$.

To this end, let $r> p'$ be arbitrary but fixed. In this case we may find $\delta<1$ (but possibly very close) so that
\[ \frac{1}{r} = \frac{1}{p'} + (\delta - 1).\]
We set $\gamma = \gamma(\delta) = (\delta - \frac{1}{2})^{-1}$, so that
\[ \gamma (\delta - \frac{1}{2}) = 1.\]
If we choose $\theta$ to satisfy
\[ \frac{1}{p} = \frac{1-\theta}{\delta^{-1}} + \frac{\theta}{2} \]
then we also have
\[ \frac{1}{r} = \frac{\theta}{\gamma}, \]
by direct computation.

The previous Lemma \ref{absint} allows us to interpolate the result.

We next turn to the case $p > 2$. With $r > p$, and $\theta$ satisfying
\[ \frac{1}{p} = \frac{\theta}{2},\]
we may write
\[ \frac{1}{r} = \frac{\theta}{ \frac{2r}{p} },\]
where $\frac{2r}{p} > 2$.

Once again, Lemma \ref{absint} completes the proof.
\end{proof}

\section{The $2$-Variation is Unbounded on $L^2$}
In this section, we shift our frame of reference from the integer setting to the torus.

For functions on the torus, $\TT$, we (abuse notation and) define the convolution operators
\[ K_N^\alpha*f = K_N*f(x) := \frac{1}{N} \sum_{n\leq N} f(x+n^2 \alpha);\]
we will take $\alpha = 2^{-R}$, where $R$ is a large parameter, and define
the $2$-variation associated to the family $\{K_N\}$ in the expected way, as
\[ \V^2(K_N*f)(x) := \sup_{(N_k) \text{ increasing}} \left( \sum_{k} |K_{N_k}*f - K_{N_{k+1}}*f|^2 \right)^{1/2}(x).\]

We prove the following
\begin{theorem}[The 2-Variation Operator is Unbounded on $L^2$]
For any $C > 0$, there exists an $f = f_C$ of $L^2$-norm one, but so that
\[ \| \V^2(K_N*f) \|_{L^2(\TT)} \geq C.\]
\end{theorem}

Here's the set-up:

Let $L$ be a large natural number, and let
\[ k_1 > k_2 > \dots > k_L\]
and
\[ j_1 < j_2 < \dots < j_L \]
be two sequences to be determined presently. We also set $R:= 2^{2^L}$.

For functions in the (finite-dimensional) span of $\{ e(2^{k_i} x) : 1 \leq i \leq L\}$
\[ g(x) = \sum_{i=1}^L b_i e(2^{k_i} x),\]
we define the partial summation operators
\[ S_m g(x):= \sum_{i=m}^L b_i e(2^{k_i} x).\]
For such $g$, we define the $2$-variation of the $\{S_m\}$ in the natural way
\[ \V^2(S_m(g))(x) := \sup_{(m_k) \text{ increasing}} \left( \sum_{k} |S_{m_k}g - S_{m_{k+1}}g|^2 \right)^{1/2}(x). \]

We use the following result of Lewko and Lewko \cite{L}:
\begin{proposition}[\cite{L} Theorem 6]
There exists some
\[ f(x) := \sum_{i=1}^L a_i e(2^{k_i}x) \]
with $\|f\|_{L^2(\TT)} = 1$, but
\[ \| \V^2( S_m(f) ) \|_{L^2(\TT)} \gtrsim \log ( \log (L)).\]
\end{proposition}

Fix this $f$; suppose we knew the following
\begin{lemma}
The error function
\[ \eta(f):= \sum_{l=1}^L | S_l(f) - K_{2^{j_l}}*f| \]
is bounded uniformly on $\TT$ (and in particular has $L^2$-norm $\lesssim 1$).
\end{lemma}

Then, by the triangle inequality, we would be able to bound from below
\[ \| \V^2(K_{2^{j_l}}*f) \|_2 \geq
\| \V^2(S_l(f)) \|_2 - \| \eta(f) \|_2 \gtrsim \log ( \log (L)) - O(1) \]
which tends to $\infty$ with $L$, which would prove our theorem.

Before beginning the proof of our (technical) lemma, we sketch out our strategy.

We consider the interactions
\[ K_{2^{j_l}}*e(2^{k_i}x) \]
in two separate regimes: when $k_i$ is ``$R$-small'' relative to $j_l$ ($ i \geq l$), and
\[ K_{2^{j_l}}*e(2^{k_i}x) ``=" e(2^{k_i}x) \]
and when
$k_i$ is ``$R$-large'' relative to $j_l$ ($ i < l$), and
\[ K_{2^{j_l}}*e(2^{k_i}x) ``=" 0. \]

More precisely, in the first regime we use Weyl's Lemma for quadratic polynomials \cite{LY}, to bound
\[
\left| \frac{1}{2^{j_l}} \sum_{n \leq 2^{j_l} } e( \frac{2^{k_i}}{2^R} \cdot n^2 ) \right| \lesssim
j_l \left( \frac{2^{k_i}}{2^R} + \frac{1}{2^{j_l}} + \frac{2^R}{2^{2j_l + k_i}} \right)^{1/2};\]
when $j_l$ is large (as it will be), and
\[ \frac{2^{k_i}}{2^R} \geq \frac{2^R}{2^{2j_l + k_i}} \]
(as it will be)
we have the approximate uniform bound
\[ \left| \frac{1}{2^{j_l}} \sum_{n \leq 2^{j_l} } e( \frac{2^{k_i}}{2^R} \cdot n^2 ) \right| \lesssim j_l \sqrt{\frac{2^{k_i}}{2^R}} + \text{ small errors}.\]

In the second regime, we simply use the mean-value theorem to estimate
\[ \left| K_{2^{j_l}}*e(2^{k_i}x) - e(2^{k_i}x) \right| \lesssim \frac{2^{k_i + 2j_l}}{2^R}. \]

Up to some careful optimization, these estimates allow us to uniformly bound

\[ \aligned
\eta(f) &= \sum_{l=1}^L | S_l(f) - K_{2^{j_l}}*f| \\
&= \sum_{l=1}^L \left| \left( \sum_{i=1}^{l-1} a_i K_{2^{j_l}}*e(2^{k_i} x) \right) -
\left( \sum_{i=l}^{L} a_i K_{2^{j_l}}*e(2^{k_i} x) - e(2^{k_i} x) \right) \right| \\
&\leq \sum_{l=1}^L \left( \sum_{i=1}^{l-1} |K_{2^{j_l}}*e(2^{k_i} x)| +
\sum_{i=l}^{L} |K_{2^{j_l}}*e(2^{k_i} x) - e(2^{k_i} x)| \right), \endaligned \]
taking into account the trivial $|a_i| \leq \sum_i |a_i|^2 = 1$.

\begin{proof}
We begin by recursively constructing our sequences $\{k_l : 1 \leq l \leq L\}$ and $\{j_l : 1 \leq l \leq L \}$:

Starting with our top terms
\[ k_L := 0, \ j_L := \frac{R-L}{2},\]
we define
\[ k_{l-1} := R - j_l, \ \text{ and } j_l := \frac{R-L-k_l}{2};\]
in particular,
\[ k_{l-1} + j_l = R, \ \text{ and } k_l + 2j_l + L = R.\]
By induction, we find
\[ k_{L-t} = \frac{2^t - 1}{2^t} (R+L) , \ \text{ and } j_{L-t} = \frac{R - (2^{t+1} -1 )L}{2^{t+1}},\]
and so
\[ k_1 = \frac{2^{L-1} - 1}{2^{L-1}} (R+L), \ \text{ and } j_1 = \frac{ R - (2^L - 1)L}{2^L}.\]

Now, with $1 \leq l \leq L$ temporarily fixed, we estimate terms in the first regime
\[
\aligned
| K_{2^{j_i}}* e(2^{k_i}x) | &\lesssim  j_l \left( \frac{2^{k_i}}{2^R} + \frac{1}{2^{j_l}} + \frac{2^R}{2^{2j_l + k_i}} \right)^{1/2} \\
&\lesssim \frac{j_l}{2^{j_l/2}} + j_l \left( \frac{2^{k_i}}{2^R} + \frac{2^R}{2^{2j_l + k_i}} \right)^{1/2} \\
&\leq \frac{j_1}{2^{j_1/2}} + j_l \sqrt{ \frac{2^{k_i}}{2^R} }, \endaligned \]
since for $i \leq l-1$, we have
\[ \frac{2^{k_i}}{2^R} \geq \frac{2^R}{2^{2j_l + k_i}},\]
by our choice
\[ k_{l-1} + j_l = R.\]

Summing $\frac{j_1}{2^{j_1/2}} + j_l \sqrt{ \frac{2^{k_i}}{2^R} }$ over $1 \leq i < l$ leads to an upper estimate of no more than
\[ L \cdot \frac{j_1}{2^{j_1/2}} + j_l \sqrt{ \frac{2^{k_1}}{{2^R}}}.\]

We next estimate terms on our second regime, $l \leq i \leq L$,
\[ \left| K_{2^{j_l}}*e(2^{k_i}x) - e(2^{k_i}x) \right| \lesssim \frac{2^{k_i + 2j_l}}{2^R} \leq \frac{2^{k_l + 2j_l}}{2^R} = \frac{1}{2^L},\]
using that $k_l + 2j_l +L = R$.
Summing over $l \leq i \leq L$ contributes a further
\[ \frac{L}{2^L},\]
so that for our given $l$, we have
\[ | S_l(f) - K_{2^{j_l}}*f| \leq L \cdot \frac{j_1}{2^{j_1/2}} + j_l \sqrt{ \frac{2^{k_1}}{{2^R}}} + \frac{L}{2^L}.\]
Finally, summing over $1 \leq l \leq L$ leads to the estimate
\[ L^2 \frac{j_1}{2^{j_1/2}} + j_L \sqrt{ \frac{2^{k_1}}{{2^R}} } + \frac{L^2}{2^L},\]
which is $O(1)$ for all $L$ sufficiently large.

The proof is complete.
\end{proof}

\begin{remark}
Although we have chosen for simplicity to work with a rational
\[ \alpha = 2^{-R}, \]
this is purely a matter of taste: setting e.g.\
\[ \alpha' = 2^{-R} + \beta \]
where (say)
\[ \beta = .\underbrace{0 \dots 0}_{\text{$2^{2^R}$ zeroes}} 1 \underbrace{0 \dots 0}_{\text{$2^{3^R}$ zeroes}} 1 \underbrace{0 \dots 0}_{\text{$2^{4^R}$ zeroes}} 1 \dots, \]
leaves our proof unchanged, since both the mean-value and Weyl estimates remain valid. The key point is that for each $\{ k_i \}$
\[ | 2^{k_i} \alpha - 2^{k_i} \alpha' | = | 2^{k_i} \beta | \ll \frac{1}{2^{2R}}.\]
\end{remark}

\begin{remark}
An interesting question concerns the behavior of the $\V^2$ operator on other $L^p$ spaces. Since our function $f_C$ has is a linear combination of lacunary frequencies, it has $L^p$ norm $\approx_p 1$ for each $1 \leq p<\infty$. The interesting question -- beyond the scope of the present paper -- is whether \cite[Theorem 6]{L} is generalizable to other $L^p$ spaces. Given the poor behavior of the $\V^2$ operator associated to the standard Birkhoff averages (cf. e.g.\ \cite{J}) we feel comfortable risking the following:

\begin{conj}
For each $1 \leq p \leq \infty$, and any integer-valued polynomial $P(n)$, the operator
\[ \V^2(M_Nf) \]
is unbounded on $L^p(X)$.
\end{conj}

\end{remark}

\end{document}